\documentclass[hidelinks]{amsart}

\usepackage[defaultlines=2,all]{nowidow} 

\usepackage{ mathtools, amssymb, stmaryrd, mathdots, bbm, mleftright, faktor }

\usepackage{ letltxmacro }    
\usepackage{ xparse }
\usepackage{ stackengine, graphicx, scalerel, etoolbox }

\usepackage[table,xcdraw,rgb]{xcolor}
\definecolor{darkazure}{HTML}{0059b3}
\definecolor{azure}{HTML}{007fff}
\definecolor{paleazure}{HTML}{a6d2ff}
\definecolor{vdarkgold}{HTML}{807100}
\definecolor{darkergold}{HTML}{b39e00}
\definecolor{darkgold}{HTML}{f2d400}
\definecolor{gold}{HTML}{ffdf00}
\definecolor{palegold}{HTML}{fff7bf}

\usepackage{ tikz, tikz-cd }
\usetikzlibrary{shapes.misc, calc, decorations.pathreplacing, patterns, arrows.meta, positioning}

\usepackage[shortlabels]{ enumitem }

\usepackage[section]{ placeins } 

\counterwithin{figure}{section}
\counterwithin{table}{section}

\usepackage[margin=0cm]{ caption }

\usepackage{ marginnote }

\makeatletter
\def\subsection{\@startsection{subsection}{2}%
  \z@{.5\linespacing\@plus.7\linespacing}{.3\linespacing}%
  {\normalfont\bfseries}}
\makeatother

\makeatletter
\def\subsubsection{\@startsection{subsubsection}{3}%
  \z@{5ex\@plus1ex}{-1ex}%
  {\normalfont\itshape}}%
\makeatother

\usepackage{ amsthm, thmtools, url }
\usepackage[linktoc=all]{ hyperref }
\usepackage[nameinlink, capitalize]{ cleveref }

\numberwithin{equation}{section}
\declaretheorem[style=plain,numberlike=equation]{theorem}
\declaretheorem[style=plain,numberlike=theorem]{lemma}

\declaretheorem[style=plain,numberlike=theorem]{proposition}

\declaretheorem[style=definition,numberlike=theorem]{definition}
\declaretheorem[style=definition,numberlike=theorem]{remark}
\declaretheorem[style=definition,numberlike=theorem]{example}

\SetEnumitemKey{smallin}{leftmargin=2.7em, labelindent=*}   
\setlist{smallin, topsep=1pt}

\crefformat{section}{\S#2#1#3}
\crefrangeformat{section}{\S#3#1#4--#5#2#6}
\crefmultiformat{section}{\S#2#1#3}{ and~\S#2#1#3}{, \S#2#1#3}{, and~\S#2#1#3}
\crefrangemultiformat{section}{\S#2#1#3}{ and~\S#2#1#3}{, \S#2#1#3}{, and~\S#2#1#3}

\crefformat{equation}{(#2#1#3)}
\crefrangeformat{equation}{(#3#1#4)--(#5#2#6)}
\crefmultiformat{equation}{(#2#1#3)}{ and~(#2#1#3)}{, (#2#1#3)}{, and~(#2#1#3)}
\crefrangemultiformat{equation}{(#2#1#3)}{ and~(#2#1#3)}{, (#2#1#3)}{, and~(#2#1#3)}

\crefname{enumi}{}{}

\makeatletter
\let\cite\relax
\DeclareRobustCommand{\cite}{%
  \let\new@cite@pre\@gobble
  \@ifnextchar[\new@cite{\@citex[]}}
\def\new@cite[#1]{\@ifnextchar[{\new@citea{#1}}{\@citex[#1]}}
\def\new@citea#1{\def\new@cite@pre{#1}\@citex}
\def\@cite#1#2{[{\new@cite@pre\space#1\if\relax\detokenize{#2}\relax\else, #2\fi}]}
\makeatother

\newcommand{\initialsspace}{0.1em}
 
\makeatletter
\newcommand{\splitlist}[1]{\@splitlist#1\@nil}
\def\@splitlist#1\@nil{%
  \if\relax\detokenize{#1}\relax
    \expandafter\@gobble
  \else
    \expandafter\@firstofone
  \fi
  {\@spl@tlist#1.\@nil}%
}
 
\def\@spl@tlist#1.#2\@nil{%
    \def\tmpA{#1}%
    \def\tmpB{#2}%
    \def\tmpP{.}%
    \ifx\tmpB\tmpP%
        #1.%
    \else{%
        \ifx\tmpA\@empty%
        \else%
                #1.\nobreak\hspace{\initialsspace}%
        \fi%
    }%
    \fi%
  \if\relax\detokenize{#2}\relax
    \expandafter\@firstoftwo
  \else
    \expandafter\@secondoftwo
  \fi
  {\unskip}%
  {\@spl@tlist#2\@nil}%
}
\makeatother

\NewDocumentCommand\set{s m}{%
    \IfBooleanTF#1%
    {\left\{ #2 \right\}}%
    {\{#2\}}%
}
\NewDocumentCommand\setbuild{s m m}{%
    \IfBooleanTF#1%
    {\ensuremath{\left\{\, #2 \, \middle| \, #3 \,\right\}}}%
    {\ensuremath{\{\, #2 \, \mid \, #3 \,\}}}%
}
\NewDocumentCommand\spangle{s m m m}{%
    \IfBooleanTF#1%
    {\ensuremath{\left\langle\, #2 \, \middle| \, #3 \,\right\rangle_{#4}}}%
    {\ensuremath{\langle\, #2 \, \mid \, #3 \,\rangle_{#4}}}%
}

\newcommand{\symdiff}{\mathbin{\triangle}} 

\newcommand{\msetbracegap}{\mkern-4.5mu}

\NewDocumentCommand\multiset{s m}{%
    \IfBooleanTF#1%
    {\left\msetbracegap\{\left\{ #2 \right\}\msetbracegap\right\}}%
    {\{\msetbracegap\{#2\}\msetbracegap\}}%
}

\newcommand{\onto}{\twoheadrightarrow}

\DeclarePairedDelimiter{\abs}{\lvert}{\rvert}

\newcommand{\N}{\mathbb{N}}

\newcommand{\iso}{\cong} 
\newcommand{\tensor}{\otimes}

\newcommand{\la}{\lambda}

\newcommand{\sstyle}{\scriptstyle}

\DeclareMathOperator{\id}{id}
\DeclareMathOperator{\im}{im}
\newcommand{\bbone}{\mathbbm{1}}
\newcommand{\one}[1]{\bbone\mkern-1mu[\,#1\,]} 
\newcommand{\blank}{{-}}    
\DeclareMathOperator{\sign}{sgn}
\newcommand{\dual}{\ast}       

\newcommand{\Y}[1]{[#1]}
\newcommand{\RSSYT}{\mathrm{RSSYT}} 
\newcommand{\CSYT}{\mathrm{CSYT}}
\newcommand{\SSYT}{\mathrm{SSYT}}
\newcommand{\col}{\mathrm{col}}
\newcommand{\bcol}[2]{\col_{#1}(#2)}
\newcommand{\ecol}[3]{\col_{#1}(#2)}
\newcommand{\row}{\mathrm{row}}
\newcommand{\brow}[2]{\row_{#1}(#2)}

\newcommand{\CPP}{\mathrm{CPP}}
\newcommand{\RPP}{\mathrm{RPP}}
\DeclareMathOperator{\rowstab}{rstab}
\newcommand{\rstab}[1]{\rowstab(#1)}

\newcommand{\colsorted}[1]{#1^{\shortdownarrow}}

\newcommand{\entryset}{\mathcal{B}}

\newcommand{\rt}[1]{[#1]} 
\newcommand{\act}[1]{|#1|} 

\newcommand{\rsym}{\operatorname{rsym}}
\newcommand{\exrsym}[1]{
    \longsubsum{\sum}
    {\tau \ssin \rcosets{\RPP(\lambda)}{\rstab{#1}}}
    {#1 \cdot \tau}
}

\newcommand{\e}{\operatorname{e}}
\newcommand{\polyt}[1]{\e(#1)} 
\newcommand{\schwa}{\operatorname{\rotatebox[origin=c]{180}{$\mathrm{e}$}}}
\newcommand{\cpt}[1]{\schwa(#1)} 

\newcommand{\copolytab}{\mathsf{copolytab}}
\newcommand{\polytab}{\mathsf{polytab}}

\newcommand{\ppa}{\cdot}        

\newcommand{\GRel}[1]{\mathsf{R}_{#1}}
\newcommand{\rotatedR}{\operatorname{\rotatebox[origin=c]{180}{$\mathsf{R}$}}}
\newcommand{\coGRel}[1]{\rotatedR_{#1}}

\newcommand{\T}[2]{#2^{\otimes\abs{#1}}}

\newcommand{\tabspace}[2]{\Sym^{#1} #2}

\newcommand{\Gar}{\mathsf{GR}}
\newcommand{\coGar}{\mathsf{G}\hspace{-1pt}\rotatedR}

\newcommand{\GRspace}[2]{\Gar^{#1}(#2)}
\newcommand{\GR}[2]{\GRspace{#1}{#2}}
\newcommand{\coGRspace}[2]{\coGar^{#1}(#2)}
\newcommand{\coGR}[2]{\coGRspace{#1}{#2}}

\newcommand{\alttabs}[2]{\Wedge^{#1'} #2}

\newcommand{\wedgemap}{\mathsf{\Lambda}}

\DeclareMathOperator{\Sym}{Sym}
\newcommand{\wwedge}{\mathchoice{{\textstyle\bigwedge}}%
    {{\bigwedge}}%
    {{\textstyle\wedge}}%
    {{\scriptstyle\wedge}}}
\DeclareMathOperator{\Wedge}{\wwedge}

\newcommand\scalesim{\scalebox{.8}{$\SavedStyle\sim$}}

\newcommand{\colless}{<_{\mathrm{c}}}
\newcommand{\colsim}{\sim_{\mathrm{c}}}
\newcommand{\colleq}{%
    \mathrel{\ensurestackMath{\ThisStyle{%
        \stackengine{-.6\LMex}{\SavedStyle{<_{\mathrm{c}}}}{%
            \rotatebox{-25}{%
                \scalesim
            }%
        }{U}{l}{F}{T}{S}
    }}}%
}

\newcommand{\rowless}{<_{\mathrm{r}}}
\newcommand{\rowsim}{\sim_{\mathrm{r}}}
\newcommand{\rowleq}{%
    \mathrel{\ensurestackMath{\ThisStyle{%
        \stackengine{-.6\LMex}{\SavedStyle{<_{\mathrm{r}}}}{%
            \rotatebox{-25}{%
                \scalesim
            }%
        }{U}{l}{F}{T}{S}
    }}}%
}

\DeclareMathOperator{\stab}{stab}
\newcommand{\stabof}[1]{\stab(#1)}

\newcommand{\lcosets}[2]{#1/#2}
\newcommand{\rcosets}[2]{#2\backslash#1}
\newcommand{\dcosets}[3]{#1\backslash #2 / #3}

\protected\def\verythinspace{%
  \ifmmode
    \mskip0.5\thinmuskip
  \else
    \ifhmode
      \kern0.08334em
    \fi
  \fi
}
\newcommand{\ssin}{\verythinspace\in\verythinspace} 

\DeclareMathOperator{\Hom}{Hom} 

\newcommand{\GL}{\mathrm{GL}}

\renewcommand{\epsilon}{\varepsilon}
\renewcommand{\phi}{\varphi}
\renewcommand{\leq}{\leqslant}

\renewcommand{\geq}{\geqslant}

\LetLtxMacro{\oldfaktor}{\faktor}
\renewcommand{\faktor}[2]{\,\oldfaktor{#1}{#2}\,}

\usepackage{letltxmacro}
\makeatletter
\let\oldr@@t\r@@t
\def\r@@t#1#2{%
\setbox0=\hbox{$\oldr@@t#1{#2\,}$}\dimen0=\ht0
\advance\dimen0-0.2\ht0
\setbox2=\hbox{\vrule height\ht0 depth -\dimen0}%
{\box0\lower0.4pt\box2}}
\LetLtxMacro{\oldsqrt}{\sqrt}
\renewcommand*{\sqrt}[2][\ ]{\oldsqrt[#1]{#2}}
\makeatother

\newlength{\truelen}
\newcommand{\padbox}[3][c]{%
    \settowidth{\truelen}{\ensuremath{#2}}%
    \ifdim\truelen < #3%
        \makebox[#3][#1]{\ensuremath{#2}}%
    \else%
        \ensuremath{#2}%
    \fi%
}

\newlength{\minlen}
\newcommand{\flexbox}[3][l]{%
    \settowidth{\minlen}{\ensuremath{#3}}%
    \padbox[#1]{#2}{\minlen}%
}

\newcommand{\longsubsum}[3]{%
    \flexbox{%
        \smashoperator[r]{#1_{#2}}\;#3%
    }{%
        {\scriptstyle#2}%
    }%
}

\usepackage{ytableau_mod}
\ytableausetup{centertableaux}
\newcommand{\ytab}[1]{%
    \ytableaushort{#1}%
}
\newcommand{\ytabsmall}[1]{%
    \ytableausetup{smalltableaux}%
    \ytableaushort{#1}%
    \ytableausetup{nosmalltableaux}%
}

\newcommand{\yrowtab}[1]{%
    \ytableausetup{tabloids}%
    \ytableaushort{#1}%
    \ytableausetup{notabloids}%
}
\newcommand{\ycoltab}[1]{%
    \ytableausetup{coltabloids}%
    \ytableaushort{#1}%
    \ytableausetup{notabloids}%
}
\newcommand{\ycoltabsmall}[1]{%
    \ytableausetup{coltabloids,smalltableaux}%
    \ytableaushort{#1}%
    \ytableausetup{notabloids,nosmalltableaux}%
}

\newcommand{\drawtwosides}[2]{\draw (#1,#2)--++(1,0)--++(0,1);} 
\newcommand{\drawthreesides}[2]{\draw (#1,#2+1)--++(0,-1)--++(1,0)--++(0,1);} 
\newcommand{\drawrthreesides}[2]{\draw (#1,#2)--++(1,0)--++(0,1)--++(-1,0);} 
\newcommand{\drawfoursideswstyle}[3]{\draw[#3] (#1,#2+1)--++(0,-1)--++(1,0)--++(0,1)--cycle;}  
\newcommand{\drawfoursides}[2]{\drawfoursideswstyle{#1}{#2}{}} 

\newcommand{\xlen}{0.95}
\newcommand{\ylen}{0.95}
\newcommand{\ABshift}{0.07,1-0.21} 
\newcommand{\fb}{0} 
\newcommand{\ABstyle}[1]{\(\mathrlap{\sstyle \scaleobj{0.95}{#1}}\)} 

\newcommand*{\semibig}[1]{\scaleobj{1.08}{#1}}
\newcommand{\semibigbackslash}{\semibig\backslash}

\newcommand{\ring}{R}

\title[An explicit construction of the Weyl module]%
{An explicit construction of the Weyl module as a quotient of symmetric tensors by dual Garnir relations}
\author{Eoghan McDowell}

\usepackage[hang,flushmargin]{footmisc} 
\usepackage{scalefnt}
\RequirePackage{fix-cm}

\makeatletter
\let\@makefnmark\relax \let\@thefnmark\relax
\def\makedetails{\footnotetext{\scalefont{0.91}%
    \ifx\@affiliation\@empty\else%
    \affiliationname\ \@affiliation \@addpunct.\ %
    \fi%
    \ifx\emails\@empty\else%
    \emailname\ \@email \@addpunct.\ %
    \fi%
    \ifx\web@page\@empty\else%
    \webpagename\ \web@page \@addpunct.%
    \fi
}}
\def\affiliationname{\textit{Affiliation}:}
\def\affiliation#1{\gdef\@affiliation{#1}}
\let\@affiliation\@empty
\def\emailname{\textit{Email}:}
\def\email#1{\gdef\@email{#1}}
\let\@email\@empty
\def\webpagename{\textit{Webpage}:}
\def\webpage#1{\gdef\web@page{#1}}
\let\web@page\@empty
\makeatother

\affiliation{University of Bristol}
\email{\href{mailto:eoghan.mcdowell@bristol.ac.uk}{eoghan.mcdowell@bristol.ac.uk}}
\webpage{\href{https://eoghanjmcdowell.com}{eoghanjmcdowell.com}}

\calclayout

\begin{document}

\begin{abstract}
The Weyl modules are the standard modules for the Schur algebra.
Their duals (the costandard modules) have well-known constructions as quotients of exterior powers and as submodules of symmetric powers.
This paper presents analogous constructions for the Weyl modules themselves, introducing relations that are a non-trivial dualisation of the Garnir relations.
More generally, our constructions describe endofunctors on the category of representations of any group.
\end{abstract}

\vspace*{0.8cm} 
\vspace*{-1cm}
\maketitle
\makedetails
\thispagestyle{empty}

\vspace*{-0.5cm}
\section{Introduction}

The \emph{Weyl modules} are the \emph{standard modules} for the Schur algebra; in particular, their heads form a complete irredundant set of simple modules in the category of polynomial representations of \(\GL_n(k)\), where \(n\in\N\) and \(k\) is an infinite field.
It is well-known how to construct the \emph{dual} of a Weyl module (a \emph{costandard module}) as a quotient of an exterior power and as a submodule of a symmetric power (see for example \cite[\S4]{green2006polynomial}, \cite[\S8]{fulton1997youngtableaux}, \cite[\S2]{boeck2021plethysms} or \cite[\S2]{mcdowell2021inverseSchur}, or \cite[Proposition~4.5.2]{donkin1998qschur} for the \(q\)-analogue).

Explicit constructions for the Weyl module itself, however, are lacking.
The purpose of this paper is to provide such constructions.
Comparisons with the existing literature are made in \Cref{remark:comparisons}.

Both the existing and new constructions are functorial: they in fact describe, for each partition \(\la\), a pair of endofunctors on the category of representations of any group over any unital commutative ring.
We therefore work in this generality throughout.
To obtain the Weyl and dual Weyl modules, apply the functors to the natural representation of \(\GL_n(k)\).

\subsection{Results}

We call the functors for constructing the dual Weyl modules the \emph{Schur endofunctors} and denote them \(\nabla^\la\).
(The Schur endofunctors are often simply called Schur functors; we use the term ``endofunctor'' to differentiate them from functors arising in Schur--Weyl duality, such as the functor \(f\) in \cite[\S6]{green2006polynomial}.)
We call the functors for constructing the Weyl modules the \emph{Weyl endofunctors} and denote them \(\Delta^\la\).

\begin{definition}
\label{def:endofunctors}
Let \(G\) be a group and let \(\ring\) be a unital commutative ring.
Let \(n \in \N\) and let \(\la\) be a partition of \(n\).
The \emph{Schur endofunctor} \(\nabla^\la\) on the category of representations of \(G\) over \(\ring\) is defined as the quotient of an exterior power by Garnir relations, or equivalently as the subspace of a symmetric power consisting of polytabloids, as described in \Cref{sec:Schur_endofunctors}:
\[
    \nabla^\la(\blank)
    \iso \faktor{\Wedge^{\la'}(\blank)}{\Gar^\la(\blank)}
    \iso \polytab(\blank).
\]
The \emph{Weyl endofunctor} \(\Delta^\la\) is defined by pre- and post-composing the Schur endofunctor with duality (as defined in \Cref{subsec:duality}):
\[
    \Delta^\la(\blank) = (\nabla^\la(\blank^\dual))^\dual.
\]
\end{definition}

This definition of the Schur endofunctor matches the presentations of Fulton's \emph{Schur module} \(E^\la\) \cite[\S8.1, Lemma 1, p.~107]{fulton1997youngtableaux}, and of Green's dual Weyl module \(D_{\la,K}\) \cite[\S4.6]{green2006polynomial}.
The definition of the Weyl endofunctor as the dual of the Schur endofunctor then guarantees that it yields the Weyl module, but does not give an explicit model for it.

After recalling familiar multilinear algebra in \Cref{sec:tableaux}, and the well-known construction of the Schur endofunctor in \Cref{sec:Schur_endofunctors}, we give our explicit constructions for the Weyl endofunctor in \Cref{sec:Weyl_endofunctors}.
In particular, we define the following spaces:
\begin{itemize}
    \item
\(\copolytab^\la(V) \subseteq \Wedge^{\la'}(V)\), the space of copolytabloids (defined in \Cref{subsec:copolytabloids});
    \item 
\(\Sym_\la(V) \subseteq V^{\otimes n}\), the space of row-symmetrised tensors (defined in \Cref{subsec:sym_and_row_tabloids_and_symmetrisations});
    \item 
\(\coGar^\la(V) \subseteq \Sym_\la(V)\), the space of dual Garnir relations (defined in \Cref{subsec:dual_Garnir_relations}).
\end{itemize}
We establish in \Cref{sec:copolytabs_obey_dual_Garnir} the following properties of these spaces.

\begin{theorem}
\label{thm:copolytabloids_obey_dual_Garnir}
Let \(G\) be a group, let \(\ring\) be a unital commutative ring, and let \(V\) be a representation of \(G\) over \(\ring\) which is free and finitely-generated as an \(\ring\)-module.
Let \(n \in \N\) and let \(\la\) be a partition of \(n\).
\begin{enumerate}[(a)]
    \item 
There is a short exact sequence
\begin{center}
\begin{tikzcd}
0 \arrow[r] & \coGar^{\lambda}(V) \arrow[r] & \Sym_\la(V) \arrow[r, "\wedgemap"] & \copolytab^\la(V) \arrow[r] & 0
\end{tikzcd}
\end{center}
where \(\wedgemap\) is the restriction of the projection map \(V^{\otimes n} \to \Wedge^{\la'}(V)\), which sends a row-symmetrised tensor labelled by a tableau \(t\) to the copolytabloid labelled by the tableau \(t\).
\item
The space \(\copolytab^\la(V)\) has \(\ring\)-basis consisting of copolytabloids labelled by semistandard tableaux.
\end{enumerate}
\end{theorem}

Finally in \Cref{sec:copolytab_dual_to_Schur} we show that our constructions indeed yield the Weyl endofunctor, as stated below.
Our proof of the first isomorphism, \(\Delta^\la(V) \iso \copolytab^\la(V)\), is independent of the proof of \Cref{thm:copolytabloids_obey_dual_Garnir}.

\begin{theorem}
\label{thm:constructions_work}
Let \(G\) be a group, let \(\ring\) be a unital commutative ring, and let \(V\) be a representation of \(G\) over \(\ring\) which is free and finitely-generated as an \(\ring\)-module.
Let \(n \in \N\) and let \(\la\) be a partition of \(n\).
Then
\[
    \Delta^\la(V) \iso \copolytab^\la(V) \iso \faktor{\Sym_\la(V)}{\coGar^\la(V)}.
\]
\end{theorem}

\subsection{Relationship to Specht modules for the symmetric group}
The well-known construction of the Specht modules due to James \cite{gdjames1978reptheorysymgroups} can be obtained from the Schur endofunctors after an easy modification (restricting to subspaces spanned by elements labelled by tableaux with all entries distinct).
This is the setting in which the term \emph{Garnir relations} (or \emph{Garnir elements}) originates, following work of Garnir \cite{garnir1950theorie}.

The modification to the Schur endofunctors in the case of the Specht modules makes duality much simpler: the dual of a Specht module can be constructed as a submodule of a \emph{twisted Young permutation module} (analogous to an exterior power) in essentially the same way as the Specht module can be constructed as a submodule of a \emph{Young permutation module} (analogous to a symmetric power) (see \cite[\S7.4]{fulton1997youngtableaux}).
Furthermore, as described in \Cref{remark:comparison_with_Specht_module_row_relations} below, the dual Garnir relations introduced in this paper reduce to relations of essentially the same form as the usual Garnir relations. 
Thus the novelty of our dual Garnir relations is for Weyl modules and other general applications of Weyl endofunctors.

\subsection{Existing constructions of Weyl modules and Weyl endofunctors}
\label{remark:comparisons}

Akin, Buchsbaum and Weyman present the Weyl endofunctor (which they call the \emph{coSchur functor}) as a quotient of a divided power (isomorphic to the space of symmetric tensors) and as a submodule of an exterior power \cite[Definition~II.1.3]{akin1982Schurfunctors}: their definition is as the image of a map which is recursively defined by composing multiplications and comultiplications, and a description of the kernel occurs only implicitly (with unspecified integral coefficients) in the sketched proof of \cite[Thoerem~II.3.16]{akin1982Schurfunctors}.
By contrast, our map \(\wedgemap \colon \Sym_\la V \to \copolytab^\la(V)\) has a straightforward definition that is easy to compute, we explicitly describe its kernel as the space of dual Garnir relations (\Cref{def:row_Garnir_relations,prop:Weyl_endofunctor_quotient_of_symmetric_power}), and the non-trivial dualisation of the usual ``straightening algorithm'' is presented in its entirety (\Cref{subsec:straightening}).
We note that the dual Garnir relations are significantly harder to identify than the usual \emph{Garnir relations} or \emph{quadratic relations} in the quotient construct of the dual Weyl module.

Carter and Lusztig, and later Green, describe the Weyl module as the set of tensors satisfying certain conditions (\cite[\S3.2]{carterlusztig1974modularreps}, \cite[(5.2b)]{green2006polynomial}) and give a basis in terms of the action of the Schur algebra on a distinguished element (\cite[\S3.3--3.5]{carterlusztig1974modularreps}, \cite[(5.3b)]{green2006polynomial}).
Their conditions require the tensors to be antisymmetric, and so this construction is isomorphic to a subspace of an exterior power.
Our construction, meanwhile, describes the elements (including a basis) explicitly in terms of the natural basis of the exterior power.
(Carter and Lusztig introduced the name \emph{Weyl module}; for further historial notes, see \cite[p.~10]{green2006polynomial}.)

The treatments of Carter and Lusztig, and of Green, also show that the Weyl module is cyclic, generated by an element that can be interpreted as a highest weight vector.
In an unpublished note, Wildon \cite{Wildon2020WeylModule} takes this as the defining property of the Weyl module, and shows that it has a basis consisting of the elements we call copolytabloids.
Our approach differs by avoiding the use of the group action (and avoiding the assumption that the field is infinite), and hence allowing us to construct the Weyl endofunctor for any group algebra.

Constructions of the Schur and Weyl endofunctors in terms of the letter place algebra were given by Kouwenhoven \cite{kouwenhoven1991SchurandWeyl}. 

This paper contains the novel constructions from Chapter~1 of the author's doctoral thesis \cite{mcdowell2021thesis}.

\section{Tableaux and multilinear algebra}
\label{sec:tableaux}

Throughout, let \(G\) be a group, let \(\ring\) be a unital commutative ring, and let \(V\) be a representation of \(G\) over \(\ring\) which is free and finitely-generated as an \(\ring\)-module.
Let \(\entryset\) be an \(\ring\)-basis for \(V\), and fix a total order on \(\entryset\).
We frequently identify \(\entryset\) with an initial segment of \(\N\) (that is, we view \(\entryset \iso \{1, 2, \ldots, |\entryset|\}\).

\subsection{Partitions and tableaux}

A \emph{partition} is a weakly decreasing finite sequence of positive integers.
The \emph{size} of a partition, denoted \(\abs{\la}\), is the sum of its parts; we say \(\la\) is a partition of \(\abs{\la}\).
If \(\lambda\) is a partition with \(l\) parts, we interpret \(\lambda_i = 0\) for \(i > l\).
The \emph{conjugate} (or \emph{transpose}) of \(\lambda\), denoted \(\lambda'\), is the partition defined by \(\lambda'_i = \abs{\setbuild{j \geq 1}{ \lambda_j \geq i}}\).

The \emph{Young diagram} of \(\lambda\) is the set \(\Y{\lambda} = \setbuild{(i,j)}{1 \leq i \leq \lambda'_1, 1 \leq j \leq \lambda_i}\), which we picture lying in the plane (with the \(x\)-direction downward and the \(y\)-direction rightward).
An element of a Young diagram is called a \emph{box}.
Let \(\brow{i}{\la}\) and \(\bcol{j}{\la}\) denote the sets of boxes in row \(i\) and column \(j\) of \(\Y{\lambda}\) respectively.

A \emph{tableau} of shape \(\lambda\) with entries in \(\entryset\) is a function \(\Y{\lambda} \to \entryset\).
The image of a box \(b \in \Y{\lambda}\) under a tableau \(t\) is the \emph{entry} of \(t\) in \(b\).
We depict a tableau \(t\) by filling the boxes in the Young diagram of \(\lambda\) with their entries in \(t\) (see \Cref{fig:tableau}).
Unless otherwise stated, all tableaux are understood to be of shape \(\la\) with entries in \(\entryset\).

\begin{figure}[bt]
    \centering
\[
t = \ytab{{1}{2}{2},{3}{3}}
\quad\implies\quad
\rt{t} = \yrowtab{{1}{2}{2},{3}{3}}, \quad
    \act{t} = \ycoltab{{1}{2}{2},{3}{3}}
\]
    \caption{A tableau of shape \((3,2)\) with entries in \(\{1,2,3\}\), and its corresponding row tabloid and column tabloid.}
    \label{fig:tableau}
\end{figure}

The group \(G\) acts on the \(R\)-module freely generated by tableaux 
by acting entrywise on each tableau and expanding the result \(R\)-multilinearly.
This behaves like the usual ``diagonal'' action of \(G\) on \(\T{\lambda}{V}\), the \(|\la|\)-fold tensor product over \(\ring\), so we identify the space of tableaux with \(\T{\lambda}{V}\).

If the entries of a tableau strictly increase down the columns, we say it is \emph{column standard};
the set of column standard tableaux of shape \(\la\) with entries in \(\entryset\) is denoted \(\CSYT_\entryset(\la)\).
If the entries of a tableau weakly increase along the rows, we say it is \emph{row semistandard};
the set of row semistandard tableaux of shape \(\la\) with entries in \(\entryset\) is denoted \(\RSSYT_\entryset(\la)\).
If a tableau is both row semistandard and column standard, we say it is \emph{semistandard};
the set of semistandard tableaux of shape \(\lambda\) with entries in \(\entryset\) is denoted \(\SSYT_\entryset(\lambda)\).

\subsection{Place permutation action on tableaux}

Given a set \(X\), we let \(S_X\) denote the group of permutations of \(X\), with permutations written on the right of their arguments.
The group \(S_{\Y{\lambda}}\) of permutations of the Young diagram acts on tableaux on the right by permuting the boxes via
\[
(t \ppa \sigma)(b) = t(b\sigma^{-1})
\]
for \(\sigma \in S_{\Y{\lambda}}\) and \(b \in \Y{\lambda}\).

Define the sets of \emph{row-preserving} and \emph{column-preserving place permutations}, subgroups of \(S_{\Y{\lambda}}\), by
\[
    \RPP(\lambda) = \prod_{i=1}^{\lambda'_1} S_{\brow{i}{\la}}
    \qquad\quad \text{and} \qquad\quad
    \CPP(\lambda) = \prod_{j=1}^{\lambda_1} S_{\bcol{j}{\la}}.
\]
Given a tableau \(t\), let \(\stabof{t} \leq S_{\Y{\la}}\) denote the stabiliser of \(t\), and let \(\rstab{t}  = \stab(t) \cap \RPP(\lambda)\) denote the \emph{row stabiliser} of \(t\).

Given a group \(S\) and subgroups \(I\) and \(J\), we write:
\begin{itemize}[leftmargin=0.6cm,label={}]
    \item
\phantom{\(J\backslash\)}\(S/I\) for the set of left cosets \(\sigma I\) of \(I\) in \(S\);
    \item 
\(\rcosets{S}{I}\)\phantom{\(/J\)} for the set of right cosets \(I \sigma\) of \(I\) in \(S\);
    \item 
\(\dcosets{I}{S}{J}\) for the set of double cosets \(I\sigma J\) of \(I\) on the left and \(J\) on the right in \(S\).
\end{itemize}
Abusing notation, we denote sets of coset representatives in the same way.

\subsection{Symmetric powers, row tabloids and row symmetrisations}
\label{subsec:sym_and_row_tabloids_and_symmetrisations}

The \emph{symmetric power}, denoted \(\Sym^r\), is the quotient of the tensor power
\[
    \Sym^r V = \faktor{V^{\tensor r}}{\spangle{
            w \cdot \sigma - w
        }{
            w \in V^{\otimes r},\, \sigma \in S_r
        }{\ring}
    }.
\]
The dual notion, which we denote \(\Sym_r\) and call the \emph{lower symmetric power}, is the subspace of symmetric tensors:
\begin{align*}
\Sym_r V
    &= \spangle*{
            \longsubsum{\sum}{
                \sigma \ssin \rcosets{S_r}{\stabof{x}}
            }{
                x \ppa \sigma
            }
        }{
            \text{\(x \in V^{\tensor r}\) is a pure tensor}
        }{\ring}
    \subseteq V^{\tensor r}.
\end{align*}

For \(\la\) a partition, we write
\[
    \Sym^\lambda V = \bigotimes_{i=1}^{\lambda'_1} \Sym^{\lambda_i} V
    \qquad\quad \text{and} \qquad\quad
    \Sym_\lambda V = \bigotimes_{i=1}^{\lambda'_1} \Sym_{\lambda_i} V.
\]
Bearing in mind our identification of \(V^{\otimes |\la|}\) with the space of tableaux, we model both symmetric powers with tabular constructions, as follows.

\begin{definition}[Row tabloid]
Let \(t\) be a tableau.
The \emph{row tabloid} of \(t\), denoted \(\rt{t}\), is the equivalence class of \(t\) under row equivalence,
identified with the image of \(t\) under the quotient map \(V^{\otimes |\la|} \to \Sym^\la(V)\).
We draw a row tabloid \(\rt{t}\) by deleting the vertical lines from a drawing of \(t\), as depicted in \Cref{fig:tableau}.
\end{definition}

\begin{definition}[Row symmetrisation]
\label{def:row_symmetrisation}
Let \(t\) be a tableau.
The \emph{row symmetrisation} of \(t\), denoted \(\rsym(t)\), is the element of \(\Sym_\la V \subseteq \T{\lambda}{V}\) given by
\[
    \rsym(t) = \exrsym{t}.
\]
\end{definition}

Observe that \(\Sym^\lambda V\) has \(\ring\)-basis \(\setbuild{\rt{t}}{t \in \RSSYT(\lambda)}\) and
\(\Sym_\lambda V\) has \(\ring\)-basis \(\setbuild{\rsym(t)}{t \in \RSSYT(\lambda)}\).

\subsection{Exterior powers and column tabloids}

The \emph{exterior power}, denoted \(\Wedge^r\), is the quotient of the tensor power 
\[
\Wedge^r V \iso \faktor{V^{\tensor r}}{\spangle{
            w \in V^{\otimes r}
        }{
            \text{\(w \ppa \tau = w\) for some transposition \(\tau \in S_r\)}
        }{\ring}
    }.
\]
This notion is self-dual: the space of antisymmetric tensors is isomorphic to the exterior power.

For \(\la\) a partition, we write
\[
    \Wedge^\lambda V = \bigotimes_{i=1}^{\lambda'_1} \Wedge^{\lambda_i} V
\]
We model the exterior power with a tabular construction, as follows.

\begin{definition}[Column tabloid]
Let \(t\) be a tableau.
The \emph{(alternating) column tabloid} of \(t\), denoted \(\act{t}\), is the signed equivalence class of \(t\) under alternating column equivalence,
identified with the image of \(t\) under the quotient map \(V^{\otimes |\la|} \to \Wedge^{\la'}(V)\).
We draw a column tabloid by deleting the horizontal lines from a drawing of the corresponding tableau, as depicted in \Cref{fig:tableau}.
\end{definition}

Observe that \(\Wedge^{\lambda'} V\) has \(\ring\)-basis \(\setbuild{\act{t}}{t \in \CSYT(\lambda)}\).

\subsection{Duality}
\label{subsec:duality}

Duality in this paper refers to the usual notion for group algebras: for \(V\) a representation of a group \(G\) over a field \(\ring\), the dual is \(V^\dual = \Hom_{\ring}(V, \ring)\), with action of the group given by \((g f)(v) = f(g^{-1} v)\) for all \(v \in V\), \(g \in G\) and \(f \in \Hom_k(V,\ring)\).

A different version of duality, called \emph{contragradient duality} by Green \cite{green2006polynomial} and denoted \(\blank^\circ\), is relevant when \(G\) is a matrix group: the group inverse operation can be replaced with the matrix transpose.
However, it can be shown that in this case there is an isomorphism  \((\nabla^\la(V^\dual))^\dual \iso  (\nabla^\la(V^\circ))^\circ\) \cite[Proposition~3.10(ii)]{mcdowell2021thesis}.

\subsection{Orders on tableaux}
\label{section:orders}

We define two orders similar to the orders defined in \cite[Definitions~3.10, 3.11 and 13.8]{gdjames1978reptheorysymgroups}, but defined on tableaux rather than tabloids, and moreover on tableaux with not-necessarily-distinct entries.
The column ordering appeared also in \cite{mcdowell2021inverseSchur}.
We write \(\ecol{j}{t}{\la}\) for the multiset of entries of \(t\) in column \(j\).

\begin{definition}
\label{def:column_ordering}
The \emph{column ordering}, \(\colless\), is the strict partial order defined as follows on the set of tableaux of shape \(\la\).
Consider tableaux \(t\) and \(u\). 
\begin{itemize}
    \item
If there is equality \(\ecol{j}{t}{\lambda} = \ecol{j}{u}{\lambda}\) (as multisets) for all \(1 \leq j \leq \lambda_1\), then \(t\) and \(u\) are \(\colless\)-incomparable, and we write \(t \colsim u\).
    \item 
Otherwise, let \(m \in \entryset\) be maximal such that there exists \(j\) such that \(m \in \ecol{j}{t}{\lambda} \symdiff \ecol{j}{u}{\lambda}\) (where \(\symdiff\) denotes the multiset symmetric difference).
Let \(j\) be minimal such that \(m \in \ecol{j}{t}{\lambda} \symdiff \ecol{j}{u}{\lambda}\).
If \(m \in \ecol{j}{t}{\lambda} \setminus \ecol{j}{u}{\lambda}\), then \(u \colless t\); conversely
if \(m \in \ecol{j}{u}{\lambda} \setminus \ecol{j}{t}{\lambda}\), then \(t \colless u\).
\end{itemize}
The tableaux \(t\) and \(u\) are \(\colless\)-incomparable if and only if \(t \colsim u\).
We write \(t \colleq u\) to mean \(t \colless u\) or \(t \colsim u\).
\end{definition}

The order is easier to interpret in the case of tableaux with all entries distinct:
to compare two such tableaux, identify the largest entry which does not appear in the same column in both tableaux, and declare the \(\colless\)-greater tableau to be the one for which this element is further left.

We make the identical definitions for the \emph{row ordering} and the symbols \(\rowless\), \(\rowsim\) and \(\rowleq\), replacing all instances of ``column'' with ``row''.
For two tableaux with all entries distinct, the \(\rowless\)-greater tableau is the one in which the largest entry which does not appear in the same row in both tableaux appears further \emph{up} the tableau.

\section{Schur endofunctors}
\label{sec:Schur_endofunctors}

We briefly recall the construction of the Schur endofunctors given in \cite[\S2]{mcdowell2021inverseSchur}, which is equivalent to the constructions of \cite[\S8.1]{fulton1997youngtableaux} and \cite[\S4.6]{green2006polynomial}.

\subsection{Polytabloids}

\begin{definition}
\label{def:polytabloid}
The \emph{polytabloid} corresponding to a tableau \(t\) is the element of \(\tabspace{\lambda}{V}\) given by
\begin{align*}
    \polyt{t} &= \sum_{\sigma \in \CPP(\lambda)} \rt{t \ppa \sigma} \sign{\sigma}.
\end{align*}
Let \(\polytab^\la(V) \subseteq \Sym^\la(V)\) denote the subspace spanned by the polytabloids.
\end{definition}

An immediate consequence of the definition of a polytabloid is that \( \polyt{t \ppa \sigma} = \polyt{t} \sign{\sigma}\) for \(\sigma \in \CPP(\lambda)\),
and that \(\polyt{t} = 0\) if \(t\) has a repeated entry in a column.
It follows that the map \(e \colon \alttabs{\lambda}{V} \to \polytab^{\lambda}{V}\) defined by \(\ring\)-linear extension of
\begin{align*}
    e \colon \act{t} &\mapsto \polyt{t}
\end{align*}
is well-defined and surjective.
It is also \(G\)-equivariant.
We thus see that \(\polytab^\lambda V\) is the quotient of \(\alttabs{\lambda}{V}\) by the kernel of \(e\).

\subsection{Garnir relations}

\begin{definition}[Garnir relations]
\label{def:Garnir_relation}
Let \(t\) be a tableau of shape \(\lambda\) with entries in \(\entryset\).
Let \(1 \leq j < j' \leq \lambda_1\), and let \(A \subseteq \bcol{j}{\lambda}\) and \(B \subseteq \bcol{j'}{\lambda}\) be such that \(\abs{A} + \abs{B} > \lambda'_j\).
Choose \(\mathcal{S}\) a set of coset representatives for \(S_{A} \times S_B\) in \(S_{A \sqcup B}\).
The \emph{Garnir relation} labelled by \((t, A, B)\) is
\[
    \GRel{(t, A, B)} = \sum_{\tau \in \mathcal{S}} \act{t \ppa \tau} \sign{\tau}.
\]
Let \(\GR{\lambda}{V}\) denote 
the \(\ring G\)-submodule of \(\alttabs{\lambda}{V}\) spanned by the Garnir relations.
\end{definition}

It can be shown that \(\GR{\la}{V} = \ker e\) (see e.g.~\cite[\S 2]{mcdowell2021inverseSchur}).
That is, there is a short exact sequence
\begin{center}
\begin{tikzcd}
0 \arrow[r] & \GR{\lambda}{V} \arrow[r] & \alttabs{\lambda}{V} \arrow[r, "e"] & \polytab^{\lambda}{V} \arrow[r] & 0.
\end{tikzcd}
\end{center}

The Schur endofunctor can now be defined as in \Cref{def:endofunctors}.

\section{Construction of the Weyl endofunctors}
\label{sec:Weyl_endofunctors}

\subsection{Copolytabloids}
\label{subsec:copolytabloids}

As in the statement of \Cref{thm:copolytabloids_obey_dual_Garnir}, let \(\wedgemap\) be the restriction of the projection map \(V^{\otimes |\la|} \to \Wedge^{\la'}V\) to the submodule \(\Sym_\la V \subseteq V^{\otimes |\la|}\).

\begin{definition}
\label{def:copolytabloid}
Let \(t\) be a tableau.
The \emph{copolytabloid} of \(t\) is the element of \(\Wedge^{\lambda'} V\) given by
\begin{align*}
\cpt{t} = \wedgemap(\rsym(t)) = \longsubsum{\sum}{
        \tau \in \rcosets{\RPP(\lambda)}{\rstab{t}}
    }{
        \act{t \ppa \tau}.
    }
\end{align*}
Let \(\copolytab^\la(V) \subseteq \Wedge^{\la'} V\) denote the \(\ring\)-subspace spanned by the copolytabloids.
\end{definition}

\begin{remark}
The symbol \(\schwa\) is a \emph{schwa}; it is a rotation of the Roman \(\mathrm{e}\).
\end{remark}

By construction, \(\copolytab^\la(V)\) is the image of \(\wedgemap\), and therefore an \(\ring G\)-module.
Given tableaux \(t\) and \(u\), we make two straightforward observations:
\begin{itemize}[leftmargin=0.6cm]
    \item if \(t,u\) differ only by permuting rows (i.e.~if \(t \rowsim u\)), then \(\cpt{t} = \cpt{u}\);
    \item if \(t\) has a repeated entry in a column, then it is \emph{not} necessarily true that \(\cpt{t} = 0\).
\end{itemize}

We note the following ``unitriangular'' property of copolytabloids.
Recall the column ordering \(\colless\) from \Cref{section:orders}.

\begin{proposition}
\label{lemma:semistandard_polysyms_are_lin_indep}
Let \(t\) be a row semistandard tableau.
Then
\[
    \cpt{t} = \act{t} + \sum_{t \verythinspace \colless \verythinspace u} m_u \act{u}
\]
for some elements \(m_u\) in the subring of \(\ring\) generated by \(1\).
In particular, the set \(\setbuild{\cpt{s}}{s \in \SSYT(\lambda)}\) is \(\ring\)-linearly independent.
\end{proposition}

\begin{proof}
Since \(t\) is row semistandard, \(t  \colleq t \ppa \tau\) for all \(\tau \in \RPP(\lambda)\), with \(t  \colsim t \ppa \tau\) if and only if \(\tau \in \rstab{t}\).
The claimed expression for \(\cpt{t}\) is then clear.
Linear independence of \(\setbuild{\cpt{s}}{s \in \SSYT(\lambda)}\) follows: the \(\colless\)-least column standard tableau whose column tabloid appears in \(\cpt{s}\) is, being \(s\) itself, distinct for each \(s \in \SSYT(\lambda)\).
\end{proof}

\subsection{Dual Garnir relations}
\label{subsec:dual_Garnir_relations}

\begin{definition}[Dual Garnir relations]
\label{def:row_Garnir_relations}
Let \(t\) be a tableau of shape \(\lambda\) with entries in \(\entryset\).
Let \(1 \leq i < i' \leq \lambda'_1\), and let \(A \subseteq \row_i(\lambda)\) and \(B \subseteq \brow{i'}{\la}\) be such that \(\abs{A} + \abs{B} > \lambda_i\).
Let \(\mathcal{T} = \setbuild{t\ppa \tau}{\tau \in S_{A \sqcup B}}\) be the set (\emph{not} multiset) of tableaux which can be obtained from \(t\) by permuting boxes in \(A \sqcup B\).
Let \(\mathcal{T}/{\rowsim}\) denote a set of equivalence class representatives for tableaux in \(\mathcal{T}\) modulo row equivalence.
The \emph{dual Garnir relation} labelled by \((t, A, B)\) is the element of \(\Sym_\lambda V\) given by
\[
\coGRel{(t, A, B)}
    = \sum_{u \in \mathcal{T}/{\rowsim}}
    \abs*{ \rstab{u} : \rstab{u} \cap (S_{A \sqcup B} \times S_{\Y{\lambda} \setminus A \sqcup B}) }
    \rsym(u).
\]
Let \(\coGR{\lambda}{V}\) denote the \(\ring\)-submodule
of \(\Sym_{\lambda}{V}\) spanned by the dual Garnir relations.
\end{definition}

In \Cref{sec:copolytabs_obey_dual_Garnir}, we show that the dual Garnir relations comprise the kernel of the surjection \(\wedgemap \colon \Sym_\la V \to \copolytab^\la(V)\).
In the remainder of the present section, we make some observations about dual Garnir relations, including giving an example to demonstrate how to compute them.

A dual Garnir relation does not depend on the choice of equivalence class representatives: in the notation of the definition, if \(u, u' \in \mathcal{T}\) are such that \(u \rowsim u'\), then clearly \(\rsym(u) = \rsym(u')\); furthermore there exists \(\sigma \in \RPP(\lambda) \cap (S_{A \sqcup B} \times S_{\Y{\lambda}\setminus A \sqcup B})\) such that \(u\sigma = u'\) and hence \(\rstab{u} = \sigma\rstab{u'}\sigma^{-1}\), and so the index of \(\rstab{u} \cap (S_{A \sqcup B} \times S_{\Y{\lambda} \setminus A \sqcup B})\) in \(\rstab{u}\) is unchanged if \(u\) is replaced with \(u'\).
The representatives must indeed be chosen from \(\mathcal{T}\), however: if \(s \not\in \mathcal{T}\), then \(u \rowsim s\) does not imply that the relevant indices are equal.

It is not immediately clear that the \(\ring\)-submodule \(\coGR{\lambda}{V}\) is an \(\ring G\)-submodule.
We do not show this fact directly, but it follows from showing \(\coGR{\lambda}{V} = \ker \wedgemap\).

A dual Garnir relation can also be expressed as a sum over double cosets.
This expression, given in \Cref{lemma:double_coset_sum_interpretation_of_row_Garnir_relations} below, is helpful for proving that the dual Garnir relations lie in the kernel of \(\wedgemap\) (\Cref{prop:row_Garnir_relations_evaluate_to_zero}), but may be more cumbersome for explicit calculations.

\begin{proposition}
\label{lemma:double_coset_sum_interpretation_of_row_Garnir_relations}
Let \(\coGRel{(t,A,B)}\) be any dual Garnir relation.
Let \[\mathcal{S} = \dcosets{(\stab(t) \cap S_{A\sqcup B})}{S_{A \sqcup B}}{(S_A \times S_B)}\] be a set of double coset representatives for \(\stab(t) \cap S_{A \sqcup B}\) on the left and \(S_{A} \times S_B\) on the right in \(S_{A \sqcup B}\).
Then
\[
\coGRel{(t, A, B)}
    = \sum_{\tau \in \mathcal{S}}
    \abs*{
        \rstab{t \ppa \tau}
        : 
        \rstab{t \ppa \tau} \cap (S_{A \sqcup B} \times S_{\Y{\lambda} \setminus A \sqcup B})
    }
    \rsym(t \ppa \tau).
\]
\end{proposition}

\begin{proof}
Both the definition of \(\coGRel{(t,A,B)}\) and the expression in the statement above can be viewed as sums over \(S_{A \sqcup B}\) modulo certain equivalence relations: in the definition, modulo equality in \(\mathcal{T}\) and row equivalence; in the claim, modulo left multiplication by \(\stab(t) \cap S_{A \sqcup B}\) and right multiplication by \(S_A \times S_B\).
Left multiplication by \(\stab(t) \cap S_{A \sqcup B}\) precisely corresponds to equality in \(\mathcal{T}\): given \(\tau, \tau' \in S_{A \sqcup B}\), we have \(t \ppa \tau = t \ppa \tau'\) if and only if \(\tau^{-1}\tau' \in \stab(t) \cap S_{A \sqcup B}\).
Right multiplication by \(S_A \times S_B\) precisely corresponds to row equivalence in \(\mathcal{T}\): given \(u, u' \in \mathcal{T}\), we have \(u \rowsim u'\) if and only if there exists \(\sigma \in S_A \times S_B\) such that \(u \ppa \sigma = u'\) (using \(\RPP(\lambda) \cap S_{A \sqcup B} = S_A \times S_B\)).
\end{proof}

We illustrate the definition of a dual Garnir relation with an example.
This example also demonstrates why our definition is the right one to make.

\begin{example}
\label{eg:row_Garnir_relations}
Suppose \(\lambda = (2, 2)\), \(\entryset = \{1,2\}\), and 
\(
    t = \ytabsmall{11,22}.
\)
Let \(A = \set{ (1,1), (1,2) }\) and \(B = \set{ (2,1) }\) (these sets of boxes are indicated in the margin).

\marginnote{
\begin{center}
\begin{tikzpicture}[x=0.7 cm, y=-0.7 cm]
\fill[paleazure] (0+\fb,0+\fb) rectangle (0+1-\fb,1-\fb);
\fill[paleazure] (1+\fb,0+\fb) rectangle (1+1-\fb,1-\fb);

\fill[palegold] (0+\fb,1+\fb) rectangle (1-\fb,2-\fb);

\drawthreesides{0}{0}           \drawtwosides{1}{0}
    \node[darkazure] at ($(0,0)+(\ABshift)$) {\ABstyle{A}};     \node[darkazure] at ($(1,0)+(\ABshift)$) {\ABstyle{A}};
\drawfoursides{0}{1};           \drawrthreesides{1}{1}
    \node[vdarkgold] at ($(0,1)+(\ABshift)$) {\ABstyle{B}};
\end{tikzpicture}
\end{center}
}[-1.4cm]%

There are three distinct tableaux obtained by the action of \(S_{A \sqcup B}\) on \(t\):
\[
    \ytab{11,22}\,,\qquad  \ytab{21,12}\,, \qquad \ytab{12,12}\,.
\]
The latter two are row equivalent, so a set of representatives is \(\mathcal{T}/{\rowsim} = \set*{ \ytabsmall{11,22}\,,\ytabsmall{21,12}}\).
The second tableau has trivial row stabiliser.
The first tableau, \(t\) itself, has row stabiliser \(S_{\brow{1}{\lambda}} \times S_{\brow{2}{\lambda}}\), of size \(4\), whose intersection with \(S_{A \sqcup B} \times S_{\Y{\lambda} \setminus A \sqcup B}\) is \(S_{\brow{1}{\lambda}}\), of size \(2\).
Thus the dual Garnir relation is
\begin{align*}
\coGRel{(t,A,B)}
    &= 2 \rsym\bigl(\,\ytabsmall{11,22}\,\bigr) + \rsym\bigl(\,\ytabsmall{21,12}\,\bigr) \\
    &= 2 \,\ytabsmall{11,22} + \ytabsmall{21,12} + \ytabsmall{12,12} + \ytabsmall{21,21} + \ytabsmall{12,21}.
\end{align*}

Alternatively we can use the expression from \Cref{lemma:double_coset_sum_interpretation_of_row_Garnir_relations}.
The group \(S_{A \sqcup B}\) is isomorphic as an abstract group to \(S_3\), and is generated by the transpositions \(\tau = \bigl( (1,1)\ (2,1) \bigr)\) and \(\omega = \bigl( (1,1)\ (1,2) \bigr)\) (these permutations are depicted in the margin).
\marginnote{
\begin{center}
\begin{tikzpicture}[x=0.7 cm, y=-0.7 cm]
\drawthreesides{0}{0}           \drawtwosides{1}{0}
\drawfoursides{0}{1}           \drawrthreesides{1}{1}
\draw [semithick, <->] (0.5,-0.2) arc (190:350:0.5) node[above, midway]{\(\omega\)};
\draw [semithick, <->] (-0.2,1.5) arc (100:260:0.5) node[left, midway]{\(\tau\)};
\end{tikzpicture}
\end{center}
}[-2.6cm]%
Note that \(\stab(t) \cap S_{A \sqcup B} = S_A \times S_B\), and \(\omega\) is the unique nontrivial element of this subgroup.
Thus there are only two double cosets: \(\set{ \id, \omega }\) and \(\set{ \tau, \tau\omega, \omega\tau, \omega\tau\omega}\).
A choice of double coset representatives is
\[
    \dcosets{\stab(t) \cap S_{A \sqcup B}}{S_{A \sqcup B}}{S_A \times S_B}
    =
    \set{ \id, \tau }.
\]
The set \(\set{t, t \ppa \tau}\) obtained from the action of these double coset representatives is precisely the above choice of representatives for \(\mathcal{T}/{\rowsim}\), and thus we obtain \(\coGRel{(t,A,B)}\) as above.

Observe that the image of \(\coGRel{(t,A,B)}\) under the map \(\wedgemap \colon \T{\lambda}{V} \onto \Wedge^{\lambda'}{V}\) is
\begin{align*}
\wedgemap(\coGRel{(t,A,B)})
    &= 2 \cpt{t} + \cpt{t \ppa \tau} \\
    &= 2 \,\ycoltabsmall{11,22} + \ycoltabsmall{21,12} + \ycoltabsmall{12,21} \\
    &= 0,
\end{align*}
as the upcoming \Cref{prop:row_Garnir_relations_evaluate_to_zero} claims.
This allows us to express the copolytabloid \(\schwa\bigl(\, \ytabsmall{21,12}\,\bigr)\) in terms of semistandard copolytabloids: \(\schwa\bigl(\, \ytabsmall{21,12}\,\bigr) = - 2\schwa\bigl(\, \ytabsmall{11,22}\,\bigr) \). This prcoess is described in general in the upcoming \Cref{lemma:standard_expression_for_row_symmetrisation}.

We now use this example to demonstrate why our definition of the dual Garnir relations is the correct definition to make.

The obvious candidate for a simpler definition is in direct analogy with the usual Garnir relations:
define \(\coGRel{(t,A,B)}^\ast\) as a sum over left coset representatives of \(S_A \times S_B\) in \(S_{A \sqcup B}\), without any coefficients appearing in the sum.
A choice of coset representatives is \(\lcosets{S_{A \sqcup B}}{S_A \times S_B} = \set{ \id, \tau, \omega\tau }\).
In our example we have
\begin{align*}
\coGRel{(t,A,B)}^\ast
    &= \rsym \Bigl(\, \ytabsmall{11,22} \,\Bigr) + 2\rsym \Bigl(\, \ytabsmall{21,12} \,\Bigr) \\
    &= \ytabsmall{11,22} + 2 \,\ytabsmall{21,12} + 2 \,\ytabsmall{12,12} + 2 \,\ytabsmall{21,21} + 2 \,\ytabsmall{12,21}
\intertext{whose image under \(\wedgemap\) is}
\wedgemap(\coGRel{(t,A,B)}^\ast)
    &= \schwa \Bigl(\, \ytabsmall{11,22} \,\Bigr) + 2\schwa \Bigl(\, \ytabsmall{21,12} \,\Bigr) \\
    &= \ycoltabsmall{11,22} + 2\,\ycoltabsmall{21,12} + 2\,\ycoltabsmall{12,21} \\
    &= -3 \,\ycoltabsmall{11,22}
\end{align*}
which is nonzero (in characteristics other than \(3\)).
Similarly it can be shown that summing over the entire group \(S_{A \sqcup B}\) also fails to yield an element of the kernel (in this example we would obtain twice the quantity above).

An alternative definition that does yield elements of the kernel is to, as above, sum over left coset representatives of \(S_A \times S_B\) in \(S_{A \sqcup B}\) (without any additional coefficients), but replace the row symmetrisation with a sum over the entire group of row preserving permutations.
Define
\[
\coGRel{(t,A,B)}^{\ast\ast}
    =
    \sum_{\tau \ssin \lcosets{S_{A \sqcup B}}{S_A \times S_B}} \,
    \longsubsum{\sum}{\sigma \in \RPP(\lambda)}{ t \ppa \tau\sigma. }
\]
It can be shown that these elements lie in the kernel of \(\wedgemap\); this proof is much simpler than that of \Cref{prop:row_Garnir_relations_evaluate_to_zero} because the sums over the row permutations do not depend on the tableaux to which they are being applied.
However, in general these elements have scalar factors, and hence cannot be used in a straightening algorithm (as described in \Cref{lemma:standard_expression_for_row_symmetrisation}) to express a copolytabloid in terms of semistandard copolytabloids.
In our example, we have
\begin{align*}
\coGRel{(t,A,B)}^{\ast\ast}
    &= \sum_{\sigma \in \RPP(\lambda)} \ytabsmall{11,22} \ppa \sigma
        \ \ +\ \ 2\!\!\!\sum_{\sigma \in \RPP(\lambda)} \ytabsmall{21,12} \ppa \sigma \\
    &= 4\,\ytabsmall{11,22} + 2 \,\ytabsmall{21,12} + 2 \,\ytabsmall{12,12} + 2 \,\ytabsmall{21,21} + 2 \,\ytabsmall{12,21} \\
    &= 2 \coGRel{(t,A,B)}.
\end{align*}
\end{example}

\begin{remark}
\label{remark:comparison_with_Specht_module_row_relations}
If \(t\) has all entries distinct, then a dual Garnir relation labelled by \(t\) is simpler to write down: the row stabiliser of any place permutation of \(t\) is trivial, so all the coefficients in the expression are \(1\), and the row symmetrisations are sums over the entire group of row preserving permutations.
Additionally, the proof of the upcoming \Cref{prop:row_Garnir_relations_evaluate_to_zero} is much easier (as noted in \Cref{eg:row_Garnir_relations}).
In this case, when \(G\) is a symmetric group and \(V\) the natural permutation representation, our dual Garnir relations become relations for the dual Specht module that are well-known \cite[Exercise~14, p.~101]{fulton1997youngtableaux}.
Similar relations hold working in the cellular basis of the dual Specht module for the Hecke algebra \cite[\S 3.2]{mathas1999IwahoriHecke}.
\end{remark}

As is the case for the usual Garnir relations, it suffices to consider a certain subset of the dual Garnir relations:
those in which the chosen rows are adjacent, and boxes are taken from the right of the upper row and the left of the lower row.
We call these relations \emph{dual snake relations} (following the use of the name ``snake relation'' in \cite{boeck2021plethysms}, given due to the shape formed by the boxes considered).
Note we permit the chosen boxes to overlap in multiple columns;
this is due to our straightening algorithm (\Cref{lemma:standard_expression_for_row_symmetrisation}) requiring the chosen boxes to contain all or none of the instances of an entry in a row.

\begin{definition}
\label{def:row_snake_relations}
A dual Garnir relation \(\coGRel{(t,A,B)}\) is called a \emph{dual snake relation} when, in the notation of \Cref{def:row_Garnir_relations}, \(i' = i+1\) and there exist \(j \leq j'\) such that \(A = \setbuild{(i,r)}{j \leq r \leq \lambda_i}\) and \(B = \setbuild{(i',r)}{1 \leq r \leq j'}\).
In this case, we may also label the dual Garnir relation by \((t, i, (j,j'))\).
\end{definition}

\section{Copolytabloids obey the dual Garnir relations}
\label{sec:copolytabs_obey_dual_Garnir}

The goal of this section is to show that the kernel of the surjection \(\wedgemap \colon \Sym_\la \to \copolytab^\la(V)\) is the space of dual Garnir relations \(\coGar^\la(V)\), and in the process identify a basis for \(\copolytab^\la(V)\).
This will establish \Cref{thm:copolytabloids_obey_dual_Garnir}.
These facts are not necessary to deduce \(\Delta^\la V \iso \copolytab^\la(V)\); the reader interested only in this isomorphism may wish to skip ahead to \Cref{sec:copolytab_dual_to_Schur}.

\subsection{Dual Garnir relations lie in the kernel}

Our strategy is to rewrite the double coset expression for \(\coGRel{(t,A,B)}\) to remove the dependence between the sums.
To this end, we first record some expressions for sets of coset representatives.

\newcommand{\crampedtimes}{{\times}}
\newcommand{\mybigbackslash}{\,\big\backslash\,}

\begin{lemma}
\label{lemma:coset_rep_expressions}
Let \(\Gamma\) be any group, and let \(I\), \(J\) and \(L\) be subgroups of \(\Gamma\).
Denote compontentwise multiplication of sets by concatenation.
The following equalities hold, interpreted as statements about choices of coset representatives.
\begin{enumerate}[(i)]
    \item\label{eq:double_coset_expansion}
\(
I \mybigbackslash \Gamma = 
\bigsqcup_{g \ssin I \semibigbackslash \Gamma \semibig/ J} \bigsqcup_{h \ssin  (I^g \cap J) \semibigbackslash J} \set{gh}.
\)
    \item\label{eq:intermediate_row_perms}
\(
I
    \mybigbackslash
    \Gamma
=
\bigl(
    I
    \mybigbackslash
    J
\bigr)
\bigl(
    J
    \mybigbackslash
    \Gamma
\bigr)
\) if \(I \leq J\).
    \item\label{eq:expand_product_of_commuting_subgroups}
\(
(L \!\cap\! (I \crampedtimes J))
    \!\mybigbackslash\!
    (I \crampedtimes J)
=
\bigl(
    L \!\cap\! I
    \mybigbackslash
    I
\bigr)
\bigl(
    L \!\cap\! J
    \mybigbackslash
    J
\bigr)
\) if \(I\), \(J\) commute and are disjoint.
\end{enumerate}
\end{lemma}

\begin{proof}
All the statements are routine exercises in bookkeeping.
\end{proof}

\newcommand{\rowperms}[1]{\mathrm{R}_{#1}}

\begin{proposition}
\label{prop:row_Garnir_relations_evaluate_to_zero}
If \(\coGRel{(t,A,B)}\) is any dual Garnir relation,
then \(\wedgemap(\coGRel{(t,A,B)}) = 0\).
\end{proposition}

\newcommand{\AandB}{C}

\renewcommand{\mybigbackslash}{\big\backslash}

\begin{proof}
For convenience, we introduce some abbreviations: \(H = \stabof{t}\), \(C = A \sqcup B\) and \(Z = \Y{\lambda}\setminus A \sqcup B\), and for \(D \subseteq \Y{\lambda}\) we write \(\rowperms{D} = \RPP(\lambda) \cap S_D\). 
We then have, for \(\tau \in S_{\Y{\lambda}}\), that \(\stab(t \ppa \tau) = H^{\tau} = \tau^{-1}H\tau\), that \(\rstab{t \ppa \tau} = H^\tau \cap \rowperms{\Y{\lambda}}\), that \(\rstab{t \ppa \tau} \cap (S_{\AandB} \times S_{\Y{\lambda} \setminus \AandB}) = H^\tau \cap (\rowperms{\AandB} \times \rowperms{Z})\), and that \(\rowperms{\AandB} = S_A \times S_B\).

For each \(\tau \in S_{\Y{\lambda}}\), we use \Cref{lemma:coset_rep_expressions}\Cref{eq:intermediate_row_perms} with
\(\Gamma = \rowperms{\Y{\lambda}}\),
\(I = H^\tau \cap (\rowperms{\AandB}\times\rowperms{Z})\) and
\(J = H^\tau \cap \rowperms{\Y{\la}}\)
to find that
\begin{align*}
\longsubsum{\sum}{\pi \ssin
    H^\tau \cap (\rowperms{\AandB} \crampedtimes \rowperms{Z}) \semibigbackslash
    \rowperms{\Y{\lambda}}
}{
    \act{t \ppa \tau\pi}
}
&=
\sum_{\phi \ssin
    H^\tau \cap (\rowperms{\AandB} \crampedtimes \rowperms{Z}) \semibigbackslash H^\tau \cap \rowperms{\Y{\lambda}}
}\,
\longsubsum{\sum}{\sigma \ssin
    H^\tau \cap \rowperms{\Y{\lambda}}
    \semibigbackslash \rowperms{\Y{\lambda}}
}{
    \act{t \ppa \tau\phi\sigma}
} \\
&= \abs*{ H^\tau \cap \rowperms{\Y{\lambda}} :  H^\tau \cap \rowperms{C} \crampedtimes \rowperms{Z} }
\,
\longsubsum{\sum}{\sigma \ssin
    H^\tau \cap \rowperms{\Y{\lambda}}
    \semibigbackslash
    \rowperms{\Y{\lambda}}
}{
    \act{t \ppa \tau\sigma}
} \\
&= \abs*{ H^\tau \cap \rowperms{\Y{\lambda}} :  H^\tau \cap \rowperms{C} \crampedtimes \rowperms{Z} }
    \wedgemap(\rsym(t \ppa \tau))
\end{align*}
where we have used that elements of \(H^\tau \cap \rowperms{\Y{\lambda}}\) fix \(t \ppa \tau\).
Note the index is precisely the index occurring in the definition of the dual Garnir relations.
Therefore the element we are required to show is zero is
\begin{align*}
\wedgemap(\coGRel{(t,A,B)})
&=\!
\sum_{\tau \ssin
    H \cap S_{C} \semibig\backslash S_{\AandB} \semibig/ \rowperms{\AandB}
}\,
\longsubsum{\sum}{\pi \ssin
    H^\tau \cap (\rowperms{\AandB} \crampedtimes \rowperms{Z}) \semibigbackslash
    \rowperms{\Y{\lambda}}
}{
    \act{t \ppa \tau\pi}
} \\
&=\!
\sum_{\tau \ssin
    H \cap S_{C} \semibig\backslash S_{\AandB} \semibig/ \rowperms{\AandB}
}\,
\sum_{\phi \ssin
    H^\tau \cap (\rowperms{\AandB} \crampedtimes \rowperms{Z}) \semibigbackslash \rowperms{\AandB} \times \rowperms{Z}
}\,
\longsubsum{\sum}{\sigma \ssin
    \rowperms{\AandB} \times \rowperms{Z}
    \semibigbackslash
    \rowperms{\Y{\lambda}}
}{
    \act{t \ppa \tau\phi\sigma}
} \\
&=
\sum_{\tau \ssin
    H \cap S_{C} \semibig\backslash S_{\AandB} \semibig/ \rowperms{\AandB}
}\,
\sum_{\chi \ssin
    H^\tau \cap \rowperms{\AandB} \semibigbackslash \rowperms{\AandB}
}\,
\sum_{\psi \ssin
    H^\tau \cap \rowperms{Z} \semibigbackslash \rowperms{Z}
}\,
\longsubsum{\sum}{\sigma \ssin
    \rowperms{\AandB} \times \rowperms{Z}
    \semibigbackslash
    \rowperms{\Y{\lambda}}
}{
    \act{t \ppa \tau\chi\psi\sigma}
}
\end{align*}
where the last two equalities hold by parts \Cref{eq:intermediate_row_perms} and \Cref{eq:expand_product_of_commuting_subgroups} of \Cref{lemma:coset_rep_expressions} respectively,
each with \(\Gamma = \rowperms{\Y{\la}}\), and
for \Cref{eq:intermediate_row_perms} with \(I = H^\tau \cap (\rowperms{\AandB} \times \rowperms{Z})\) and \(J = \rowperms{\AandB} \times \rowperms{Z}\), and
for \Cref{eq:expand_product_of_commuting_subgroups} with \(I = \rowperms{\AandB}\), \(J = \rowperms{Z}\) and \(L = H^\tau \cap \rowperms{\Y{\la}}\).

Now we combine terms using \Cref{lemma:coset_rep_expressions}\Cref{eq:double_coset_expansion} with \(\Gamma = S_C\), \(I = H \cap S_C\) and \(J = \rowperms{\AandB}\).
We see that the final line above becomes
\begin{align*}
\wedgemap(\coGRel{(t,A,B)})
&=
\sum_{\tau \ssin
    H \cap S_{C} \semibig\backslash S_{\AandB}
}\,
\sum_{\psi \ssin
    H^{\tau} \cap \rowperms{Z} \semibigbackslash \rowperms{Z}
}\,
\longsubsum{\sum}{\sigma \ssin
    \rowperms{\AandB} \cap \rowperms{Z}
    \semibigbackslash
    \rowperms{\Y{\lambda}}
}{
    \act{t \ppa \tau\psi\sigma}.
}
\end{align*}
Notice that, because the boxes of \(Z\) are fixed by \(\tau \in S_{A \sqcup B}\), we have that \(H^{\tau} \cap \rowperms{Z} = H \cap \rowperms{Z}\) is independent of \(\tau\).
The rightmost sum above also has indexing set independent of \(\tau\), and both of these indexing sets are subsets of \(\rowperms{\Y{\lambda}}\), so it suffices to show that \(\sum_{\tau \ssin
    H \cap S_{C} \semibig\backslash S_{\AandB}} \act{t \ppa \tau\sigma} = 0\)
for all \(\sigma \in \rowperms{\Y{\lambda}}\).

Fix \(\sigma \in \rowperms{\Y{\lambda}}\).
Recall from the definition of a dual Garnir relation that \(A \subseteq \brow{i}{\lambda}\) and \(B \subseteq \brow{i'}{\lambda}\) for some \(1 \leq i < i' \leq \lambda_1\), and that \(\abs{A} + \abs{B} > \lambda_i\).
Thus by the pigeonhole principle there exists a column containing both a box in \(A\) and a box in \(B\). 
Moreover the same claim holds if we act by \(\sigma\) first; that is, there exist \(a \in A\), \(b \in B\) and \(1 \leq j \leq \lambda_1\) such that \(a\sigma = (i,j)\) and \(b\sigma = (i',j)\).
Let \(\omega = ( a \ b ) \in S_{A \sqcup B}\), and note that \(\omega^\sigma  = \sigma^{-1} \omega \sigma = (a\sigma\ b\sigma) \in \CPP(\lambda)\).

Define an action of \(\omega\) on the set of cosets \(\rcosets{S_{A \sqcup B}}{H \cap S_{A \sqcup B}}\) by right multiplication.
If \(\tau \in \rcosets{S_{A \sqcup B}}{H \cap S_{A \sqcup B}}\) is in an orbit of size \(1\), then \(t \ppa \tau = t \ppa \tau\omega\) and hence \(t \ppa \tau\sigma = t \ppa \tau\sigma\omega^\sigma\).
But then \(t \ppa \tau\sigma\) has the same entries in \(a\sigma\) and \(b\sigma\); since these are in the same column, this implies \(\act{t \ppa \tau\sigma} = 0\).
If \(\set{\tau, \tau\omega}\) is an orbit of size \(2\), then since \(\omega^\sigma \in \CPP(\lambda)\) we have \(\act{t \ppa \tau\omega\sigma} = \act{t \ppa \tau\sigma\omega^\sigma} = -\act{t \ppa \tau\sigma}\), and so the contributions to the sum of these orbits cancel out.
Thus the entire sum is zero, as required.
\end{proof}

\subsection{A straightening algorithm}
\label{subsec:straightening}

Recall we defined the \emph{dual snake relations} as a special set of dual Garnir relations in \Cref{def:row_snake_relations}.

\newcommand{\lowerj}{j}
\newcommand{\upperj}{j'}
\newcommand{\midj}{j_0}

\begin{lemma}
\label{lemma:ordered_expression_for_row_snake_relation}
Let \(t\) be a row semistandard tableau, and suppose \(i\) and \(\lowerj \leq \upperj\) are such that there exists \(\lowerj \leq \midj \leq \upperj\) such that:
\begin{itemize}
\item
\(t(i,\midj) \geq t(i{+}1,\midj)\);
\item
\(t(i, \lowerj) = t(i,\midj)\) and \(t(i+1, \upperj) = t(i+1, \midj)\);
\item
\(t(i,\lowerj{-}1) < t(i,\lowerj)\) (or \(\lowerj = 1\)) and \(t(i{+}1,\upperj) < t(i{+}1,\upperj{+}1)\) (or \(\upperj = \lambda_{i+1}\)).
\end{itemize}
Then
\[
    \coGRel{(t,i,(\lowerj,\upperj))} = \rsym(t) + \sum_{\substack{u \rowless t}} m_u \rsym(u)
\]
for some elements \(m_u\) in the subring of \(\ring\) generated by \(1\).
\end{lemma}

\tikzset{
  symbol/.style={
    draw=none,
    every to/.append style={
      edge node={node [sloped, allow upside down, auto=false]{$#1$}}}
  }
}
\newcommand{\smalldots}{\raisebox{0.95pt}{\scalebox{0.75}{$\cdots$}}}

\newcommand{\upleq}{\rotatebox{90}{\(\leq\)}}

\begin{proof}
The diagram below illustrates the hypotheses satisfied by the sets \(A = \setbuild{(i,r)}{\lowerj \leq r \leq \lambda_i}\) and \(B = \setbuild{(i+1,r)}{1 \leq r \leq \upperj}\) defining the Garnir relation (these sets are shaded in the diagram).

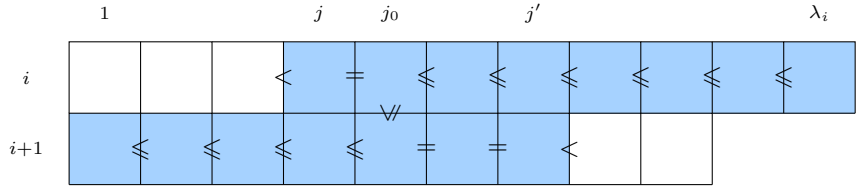
\begin{center}
\begin{tikzpicture}[x=\xlen cm, y=-\ylen cm]

\fill[paleazure] (0+\fb,1+\fb) rectangle (1+0-\fb,2-\fb);
\fill[paleazure] (1+\fb,1+\fb) rectangle (1+1-\fb,2-\fb);
\fill[paleazure] (2+\fb,1+\fb) rectangle (1+2-\fb,2-\fb);
\fill[paleazure] (3+\fb,1+\fb) rectangle (1+3-\fb,2-\fb);
\fill[paleazure] (4+\fb,1+\fb) rectangle (1+4-\fb,2-\fb);
\fill[paleazure] (5+\fb,1+\fb) rectangle (1+5-\fb,2-\fb);
\fill[paleazure] (6+\fb,1+\fb) rectangle (1+6-\fb,2-\fb);

\fill[paleazure] (3+\fb,0+\fb) rectangle (1+3-\fb,1-\fb);
\fill[paleazure] (4+\fb,0+\fb) rectangle (1+4-\fb,1-\fb);
\fill[paleazure] (5+\fb,0+\fb) rectangle (1+5-\fb,1-\fb);
\fill[paleazure] (6+\fb,0+\fb) rectangle (1+6-\fb,1-\fb);
\fill[paleazure] (7+\fb,0+\fb) rectangle (1+7-\fb,1-\fb);
\fill[paleazure] (8+\fb,0+\fb) rectangle (1+8-\fb,1-\fb);
\fill[paleazure] (9+\fb,0+\fb) rectangle (1+9-\fb,1-\fb);
\fill[paleazure] (10+\fb,0+\fb) rectangle (1+10-\fb,1-\fb);

    \node at (0.5,-0.4) {\(\sstyle 1\)};
    \node at (3.5,-0.4) {\(\sstyle \lowerj\)};
    \node at (4.5,-0.4) {\(\sstyle \vphantom{j}\smash{\midj}\)};
    \node at (6.5,-0.4) {\(\sstyle \upperj\)};
    \node at (10.5,-0.4) {\(\sstyle \lambda_i\)};
    \node at (-0.6,0.5) {\(\sstyle i\)};
    \node at (-0.6,1.5) {\(\sstyle i+1\)};
    
\drawthreesides{0}{0}
    \drawtwosides{1}{0}\drawtwosides{2}{0}\drawtwosides{3}{0}\drawtwosides{4}{0}\drawtwosides{5}{0}
    \drawtwosides{6}{0}\drawtwosides{7}{0}\drawtwosides{8}{0}\drawtwosides{9}{0}\drawtwosides{10}{0}
\drawthreesides{0}{1}
    \drawtwosides{1}{1}\drawtwosides{2}{1}\drawtwosides{3}{1}\drawtwosides{4}{1}\drawtwosides{5}{1}
    \drawtwosides{6}{1}\drawtwosides{7}{1}\drawtwosides{8}{1}
\draw (0,2)--(9,2);
\draw (9,1)--(11,1);

\node at (1,1.5) {\(\leq\)};
    \node at (2,1.5) {\(\leq\)};
    \node at (3,1.5) {\(\leq\)};
    \node at (4,1.5) {\(\leq\)};
    \node at (5,1.5) {\(=\)};
    \node at (6,1.5) {\(=\)};
    \node at (7,1.5) {\(<\)};
\node at (3,0.5) {\(<\)};
    \node at (4,0.5) {\(=\)};
    \node at (5,0.5) {\(\leq\)};
    \node at (6,0.5) {\(\leq\)};
    \node at (7,0.5) {\(\leq\)};
    \node at (8,0.5) {\(\leq\)};
    \node at (9,0.5) {\(\leq\)};
    \node at (10,0.5) {\(\leq\)};
\node at (4.5,1) {\(\upleq\)};
\end{tikzpicture}
\end{center}

In particular all the boxes in \(B\) contain entries less than or equal to all the entries in boxes in \(A\).
Therefore for any \(\tau \in S_{A \sqcup B}\),
we have \(t \ppa \tau \rowleq t\), and furthermore row equivalence \(t \ppa \tau \rowsim t\) holds if and only if \(\tau \in (\stab(t)\cap S_{A \sqcup B})(S_{A} \times S_B)\).
Thus the row symmetrisation \(\rsym(t)\) appears precisely once in \(\coGRel{(t,i,(j,j'))}\), with coefficient
\(\abs*{
    \rstab{t}
:
    \rstab{t} \cap (S_{A \sqcup B} \times S_{\Y{\lambda} \setminus A \sqcup B})
}\).
But by the assumptions, as displayed above, every entry in row \(i\) of \(t\) occurs either only in \(A\) or only in \(\brow{i}{\lambda} \setminus A\), and likewise every entry in row \(i+1\) of \(t\) occurs either only in \(B\) or only in \(\brow{i+1}{\lambda} \setminus B\).
Then  \(\rstab{t} \subseteq S_{A \sqcup B} \times S_{\Y{\lambda} \setminus A \sqcup B}\), so this coefficient is \(1\).
\end{proof}

\begin{lemma}
\label{lemma:standard_expression_for_row_symmetrisation}
Let \(t\) be a tableau.
Then there exists some \(\ring\)-linear combination \(\gamma\) of dual snake relations (with coefficients in the subring of \(\ring\) generated by \(1\)) such that
\[
    \rsym(t) + \gamma = \sum_{s \in \SSYT(\lambda)} a_s \rsym(s)
\]
for some elements \(a_s\) in the subring of \(\ring\) generated by \(1\) (which may be all zero).
\end{lemma}

\begin{proof}
Without loss of generality, we may assume \(t\) is row semistandard.
If \(t\) is also column standard, we are done.
Otherwise, choose a box \((i,j_0)\) such that \(t(i+1,j_0) \leq t(i,j_0)\).
Then pick \(j \leq j_0 \leq j'\) such that \(t(i, j{-}1) < t(i,j) = t(i,j_0)\) (or \(j=1\)) and \(t(i{+}1, j_0) = t(i{+}1, j') < t(i{+}1, j'{+}1)\) (or \(j'=\lambda_{i+1}\)).
By \Cref{lemma:ordered_expression_for_row_snake_relation} we have that \(\coGRel{(t,i,(j,j'))} = \rsym(t) + \sum_{u \rowless t} m_u \rsym(u)\) for some elements \(m_u\) in the subring of \(\ring\) generated by \(1\).
Then \(\rsym(t) - \coGRel{(t,i,j)}\) is a linear combination of row symmetrisations of tableaux which precede \(t\) in the row ordering.
The lemma follows by induction.
\end{proof}

\begin{proposition}
The module \(\copolytab^\la(V)\) has \(\ring\)-basis \(\setbuild{\cpt{s}}{s \in \SSYT(\lambda)}\).
\end{proposition}

\begin{proof}
Linear independence was shown in \Cref{lemma:semistandard_polysyms_are_lin_indep}.
Applying \(\wedgemap\) to the expression given in \Cref{lemma:standard_expression_for_row_symmetrisation}, and using \Cref{prop:row_Garnir_relations_evaluate_to_zero} to find \(\wedgemap(\gamma)=0\), shows that any copolytabloid can be written as an \(\ring\)-linear combination of semistandard copolytabloids.
\end{proof}

\begin{remark}
That the semistandard copolytabloids span \(\copolytab^\la(V)\) can also be deduced without using the results of this section, and instead using \Cref{thm:copolytab_dual_to_Schur} below and the fact that \(\dim_{\ring} \Delta^\lambda V = \dim_{\ring} \nabla^\lambda V = \abs{\SSYT_\entryset(\lambda)}\).
\end{remark}

\begin{proposition}
\label{prop:Weyl_endofunctor_quotient_of_symmetric_power}
There is equality \(\ker \wedgemap = \coGR{\lambda}{V}\) (and consequently \(\coGR{\lambda}{V}\) is an \(\ring G\)-module), and hence there is an \(\ring G\)-isomorphism
\[
    \copolytab^\lambda(V) \iso \faktor{\Sym_{\lambda} V }{\coGR{\lambda}{V}}.
\]
\end{proposition}

\begin{proof}
From \Cref{prop:row_Garnir_relations_evaluate_to_zero}, we have that \(\coGR{\lambda}{V} \subseteq \ker \wedgemap\).
It therefore suffices to show that the dual snake relations span \(\ker \wedgemap\).

Let \(\kappa \in \ker \wedgemap\).
By \Cref{lemma:standard_expression_for_row_symmetrisation} there exists an \(\ring\)-linear combination \(\gamma\) of dual snake relations such that
\begin{align*}
    \kappa + \gamma &= \sum_{s \in \SSYT(\lambda)} \alpha_s \rsym(s)
\intertext{for some elements \(\alpha_s \in \ring\).
Applying \(\wedgemap\) and using that \(\coGR{\lambda}{V} \subseteq \ker \wedgemap\), we find}
    0 &= \sum_{s \in \SSYT(\lambda)} \alpha_s \cpt{s}.
\end{align*}
But the semistandard copolytabloids are \(\ring\)-linearly independent by \Cref{lemma:semistandard_polysyms_are_lin_indep}, so \(\alpha_s = 0\) for all \(s\).
Hence \(\kappa = - \gamma\) is in the span of the dual snake relations.
\end{proof}

The statements of \Cref{thm:copolytabloids_obey_dual_Garnir} are now all proved.

\section{Space of copolytabloids is dual to the Schur endofunctor}
\label{sec:copolytab_dual_to_Schur}

We prove the first isomorphism in \Cref{thm:constructions_work}, verifying that the space of copolytabloids indeed models the Weyl endofunctor.
A special case of the calculation in the proof below, for the tableau \(t\) defined by \(t(i,j) = i\), was given by Wildon \cite[\S 3.2]{Wildon2020WeylModule}.

\begin{theorem}
\label{thm:copolytab_dual_to_Schur}
There is an isomorphism
\[
\Delta^\la V  \iso  \copolytab^\la(V).
\]
\end{theorem}

\begin{proof}
We view \(\Delta^\la V = (\nabla^\lambda V^\dual)^\dual\) as a submodule of \((\Wedge^{\lambda'} V^\dual)^\dual\) via the dual of the surjective map \(e \colon \Wedge^{\lambda'}V^\dual \to \nabla^\lambda V^\dual\); that is, by the injective map \(e^\dual \colon (\nabla^\lambda V^\dual)^\dual \to (\Wedge^{\lambda'} V^\dual)^\dual\) defined by \(e^\dual(\theta)(x) = \theta(e(x))\) for \(\theta \in (\nabla^\lambda V^\dual)^\dual\) and \(x \in \Wedge^{\lambda'} V^\dual\).
Meanwhile we view \((\Wedge^{\lambda'} V^\dual)^\dual\) as isomorphic to \(\Wedge^{\lambda'}V\) via the isomorphism \(\psi\) defined by \( (v_{i_1}^\dual \wedge \cdots \wedge v_{i_r}^\dual)^\dual \mapsto v_{i_1} \wedge \cdots \wedge v_{i_r}\) (see \cite[Proposition~3.3]{mcdowell2021thesis} for the routine verification that this is an isomorphism).
Thus it suffices to show \(\im(\psi e^\dual) = \copolytab^\la(V)\).

Let \(\entryset^\dual\) denote the \(\ring\)-basis of \(V^\dual\) dual to \(\entryset\), and, given a tableau \(t\) with entries in \(\entryset\), let \(t^\dual\) denote the tableau obtained from \(t\) by replacing each entry from \(\entryset\) with its dual from \(\entryset^\dual\).
Let \(\rt{t^\dual}^\dual\) denote the element dual to \(\rt{t^\dual}\) in the basis of \((\tabspace{\lambda}{V^\dual})^\dual\) dual to \(\setbuild{\rt{s}}{s \in \RSSYT_{\entryset^\dual}(\lambda)}\).
Since \(\nabla^\lambda V^\dual \subseteq \Sym^\lambda V^\dual\), we can restrict a function \(\rt{t^\dual}^\dual \in (\Sym^\la V^\dual)^\dual\) to \(\nabla^\lambda V^\dual\), and moreover all elements of \((\nabla^\lambda V^\dual)^\dual\) can be obtained this way.
Thus \(\Delta^\la V\) is \(\ring\)-spanned by \(\setbuild{ \rt{t^\dual}^\dual}{ t \in \RSSYT_\entryset(\la)}\), and it suffices to show that \(\psi e^\dual(\rt{t^\dual}^\dual) \in \copolytab^\la(V)\) for all \(t \in \RSSYT_\entryset(\la)\).

Fix \(t \in \RSSYT_\entryset(\lambda)\), and we will show that \(\psi e^\dual(\rt{t^\dual}^\dual) = \cpt{t}\).
In what follows, given a statement \(\mathsf{P}\),
we set \(\one{\mathsf{P}} = 1\) if \(\mathsf{P}\) is true and 
        \(\one{\mathsf{P}} = 0\) if \(\mathsf{P}\) is false.
For any \(u \in \CSYT_{\entryset^\dual}(\lambda)\) we have
\begin{align*}
e^\dual ( \rt{t^\dual}^\dual )( \act{u} )
    &= \rt{t^\dual}^\dual( e(u) ) \\
    &= \rt{t^\dual}^\dual \sum_{\sigma \in \CPP(\lambda)} \sign(\sigma) \rt{u \ppa \sigma} \\
    &= \sum_{\sigma \in \CPP(\lambda)} \sign(\sigma) \one{\rt{u\ppa \sigma} = \rt{t^\dual}} \\
    &= \sum_{\tau \in \rcosets{\RPP(\lambda)}{\rstab{t^\dual}}}
        \sum_{\sigma \in \CPP(\lambda)}
            \sign(\sigma)
            \one{t^\dual \ppa \tau = u \ppa \sigma}
\end{align*}
where the last equality holds as there is at most one element \(\tau \in \rcosets{\RPP(\lambda)}{\rstab{t^\dual}}\) such that \(t^\dual \ppa \tau = u \ppa \sigma\), and such an element exists if and only if \(t^\dual\) and \( u \ppa \sigma\) have the same multisets of entries in each row (that is, 
if and only if \(\rt{t^\dual} = \rt{u \ppa \sigma}\)).

We employ a similar argument to collapse the sum over \(\CPP(\lambda)\).
Since \(u\) is column standard and hence has distinct entries within a column, there is at most one element \(\sigma \in \CPP(\lambda)\) such that \(t^\dual \ppa \tau = u \ppa \sigma\), and such an element exists if and only if \(t^\dual \ppa \tau\) and \(u\) have the same multisets of entries in each column. 
Write \(\colsorted{s}\) for the unique column semistandard tableau obtained from a tableau \(s\) by sorting all the columns into ascending order;
thus \(t^\dual \ppa \tau\) and \(u\) have the same multisets of entries if and only if \(\colsorted{(t^\dual \ppa \tau)} = u\).
Defining \(\sign(s \mapsto \colsorted{s})\) to be the sign of the unique column-preserving permutation which makes \(s\) into a column standard tableau (if it exists; defining \(\sign(s \mapsto \colsorted{s}) = 0\) if \(s\) does not have distinct column entries),
the expression above becomes
\begin{align*}
e^\dual ( \rt{t^\dual}^\dual )( \act{u} )  
    &= \longsubsum{\sum}{\tau \in \rcosets{\RPP(\lambda)}{\rstab{t^\dual}}}{
        \one{\colsorted{(t^\dual \ppa \tau)} = u} \sign( t^\dual \ppa \tau \mapsto \colsorted{(t^\dual \ppa \tau)} )
    }.
\end{align*}
Thus we have
\begin{align*}
e^\dual ( \rt{t^\dual}^\dual )
    &= \longsubsum{\sum}{\tau \in \rcosets{\RPP(\lambda)}{\rstab{t^\dual}}}{
        \act{\colsorted{(t^\dual \ppa \tau)}}^\dual \sign( t^\dual \ppa \tau \mapsto \colsorted{(t^\dual \ppa \tau)} )
    },
\end{align*}
where for a column standard tableau \(u\), we denote by \(\act{u}^\dual\) the element dual to \(\act{u}\) in the basis of \((\alttabs{\lambda}{V^\dual})^\dual\) dual to  \(\setbuild{\act{s}}{s \in \CSYT_{\entryset^\dual}(\lambda)}\).
Applying \(\psi\) we find that
\begin{align*}
\psi(e^\dual( \rt{t^\dual}^\dual ))
    &= \longsubsum{\sum}{\tau \in \rcosets{\RPP(\lambda)}{\rstab{t}}}{
        \act{\colsorted{(t \ppa \tau)}} \sign( t \ppa \tau \mapsto \colsorted{(t \ppa \tau)} )
    } \\
    &= \longsubsum{\sum}{\tau \in \rcosets{\RPP(\lambda)}{\rstab{t}}}{
        \act{t \ppa \tau}
    } \\
    &= \cpt{t}. \qedhere
\end{align*}
\end{proof}

\begin{remark}
The map \(\psi e^\dual\) in \Cref{thm:copolytab_dual_to_Schur} does \emph{not} send \(\polyt{t^\dual}^\dual\), the element dual to a polytabloid, to the copolytabloid \(\cpt{t}\).
Furthermore, it is \emph{not} the case that the basis \(\setbuild{\cpt{s}}{s \in \SSYT_\entryset(\lambda)}\) is dual to the basis \(\setbuild{\polyt{s}}{s \in \SSYT_{\entryset^\dual}(\lambda)}\).
\end{remark}

The second isomorphism in \Cref{thm:constructions_work} follows from \Cref{thm:copolytabloids_obey_dual_Garnir} (proved in the previous section), and both our main theorems are established.

\section*{Acknowledgements}

\noindent
The author is grateful to Mark Wildon for comments on an earlier version of this paper.

\bibliographystyle{alpha}
\bibliography{references}

\end{document}